\documentclass[11pt]{article}

\usepackage{amsmath, amssymb}
\usepackage{tikz}

\usepackage{authblk}
\usepackage{geometry}
\usepackage{hyperref}
\usepackage{graphicx}
\usepackage{epstopdf}
\usepackage{mathrsfs}
\usepackage{amsfonts}
\usepackage{amscd,amsthm,bm,psfrag}
\usepackage[numbers,sort&compress]{natbib}

\makeatletter
\renewcommand{\@seccntformat}[1]{{\csname the#1\endcsname}{\normalsize.}\hspace{.5em}}

\makeatother

\numberwithin{equation}{section}

\def \[{\begin{equation}}
\def \]{\end{equation}}

\newtheorem{thm}{Theorem}[section]

\newtheorem{lem}[thm]{Lemma}

\newtheorem{prop}[thm]{Proposition}
\newtheorem{rem}[thm]{Remark}
\newtheorem{con}[thm]{Conjecture}
\newtheorem{prob}[thm]{Problem}
\newtheorem{cla}{Claim}

\newtheorem*{thm*}{Theorem}
\newtheorem*{prop*}{Proposition}

\newtheorem{que}[thm]{Question}

\newcommand\ex{\ensuremath{\mathrm{ex}}}

\newcommand\cF{{\mathcal F}}
\newcommand\cG{{\mathcal G}}
\newcommand\cH{{\mathcal H}}

\newcommand\cL{{\mathcal L}}

\begin{document}
\title{
{Counting sunflowers in hypergraphs with bounded matching number and Erd\H{o}s Matching Conjecture in the $(t,k)$-norm} 
}

\author[a,b,*]{Junpeng Zhou\,$^{\rm a,b}$, \ Xiying Yuan\,}

\affil[a]{\small \,Department of Mathematics, Shanghai University, Shanghai 200444, PR China}
\affil[b]{\small \,Newtouch Center for Mathematics of Shanghai University, Shanghai 200444, PR China}

\date{}

\maketitle

\footnotetext{*\textit{Corresponding author}.}
\footnotetext{
Email addresses: \texttt{junpengzhou@shu.edu.cn} (J. Zhou), 
\texttt{xiyingyuan@shu.edu.cn} (X. Yuan).}





\begin{abstract}
It is well known that Erd\H{o}s Matching Conjecture concerns the maximum number of hyperedges in an $r$-uniform hypergraph with bounded matching number. As a generalization, it is natural to ask for the maximum number of copies of subhypergraphs. Given integers $r\geq2$ and $k\ge 1$, let $S_{r-1,k}^r$ denote the $r$-uniform hypergraph with hyperedges $\{e_1, \dots, e_k\}$ such that there exists an $(r-1)$-set $T$ with $e_i \cap e_j = T$ for $1\le i < j \le k$. We determine the maximum number of copies of $S_{r-1,k}^r$ in an $r$-uniform hypergraph with bounded matching number, and characterize all extremal hypergraphs. An interesting phenomenon is that the extremal numbers and extremal hypergraphs are exactly the same for all $k\ge 1$.

Our main tool is the shifting method. By establishing an injection, we prove that the shifting operation does not decrease the number of copies of $S_{r-1,k}^r$ for all $k\geq1$, thereby answering a question raised by Wang and Peng (EJC, 2026). Moreover, we present a counting method for estimating the number of copies of $S_{r-1,k}^r$ in arbitrary $r$-uniform hypergraphs.

Counting the number of copies of $S_{r-1,k}^r$ in $r$-uniform hypergraphs is closely related to Tur\'{a}n problems in the $(r-1,k)$-norm proposed by Chen, Il'kovi\v{c}, Le\'{o}n, Liu and Pikhurko. The $(r-1,k)$-norm of an $r$-uniform hypergraph $\mathcal{H}$ is the sum of the $k$-th power of the degrees $d_{\mathcal{H}}(T)$ over all $(r-1)$-subsets $T \subseteq V(\mathcal{H})$. Combining our established result with that of Frankl (JCTA, 2013), and utilizing the Newton expansion of powers and Stirling numbers of the second kind, we show that Erd\H{o}s Matching Conjecture in the $(r-1,k)$-norm holds, which generalizes the result of Brooks and Linz concerning the $(r-1,2)$-norm case. As a consequence, we obtain a version of the classical Erd\H{o}s--Ko--Rado theorem in the $(r-1,k)$-norm.
\end{abstract}



\section{\normalsize Introduction}
An \textit{$r$-uniform hypergraph} ($r$-graph for short) $\cH=(V(\cH),E(\cH))$ consists of a vertex set $V(\cH)$ and a hyperedge set $E(\cH)$, where each hyperedge in $E(\cH)$ is an $r$-subset of $V(\cH)$. The size of $E(\cH)$ is denoted by $e(\cH)$. Throughout this paper, we identify hypergraphs with their hyperedge sets. Let $\binom{V}{r}$ denote the family of all $r$-subsets of $V$, i.e., the complete $r$-graph on $V$. For positive integers $n\geq m$, we use $[n]$ and $[m,n]$ for the sets $\{1,2,\dots,n\}$ and $\{m,m+1,\dots,n\}$, respectively.

Let $\mathcal{F}$ be an $r$-graph. An $r$-graph $\cH$ is \textit{$\mathcal{F}$-free} if $\cH$ does not contain $\mathcal{F}$ as a subhypergraph. The \textit{Tur\'{a}n number} of $\mathcal{F}$, denoted by ${\rm{ex}}_r(n,\mathcal{F})$, is the maximum number of hyperedges in an $n$-vertex $\mathcal{F}$-free $r$-graph.
A classical result in extremal graph theory is Erd\H{o}s--Gallai theorem \cite{EGa}, which determines the Tur\'{a}n number of $M_{s}$, where $M_{s}$ denotes a matching of size $s$, i.e., the graph consisting of $s$ independent edges.

Given two graphs $H$ and $F$, Alon and Shikhelman \cite{AlSh} initiated the systematic study of the problem of maximizing the number of copies of $H$ in an $n$-vertex $F$-free graph, which is often called the \textit{generalized Tur\'{a}n problem}. Subsequently, Wang \cite{Wang} determined the maximum number of cliques in an $n$-vertex $M_{s}$-free graph. For a very recent survey on generalized Tur\'{a}n problems, one can refer to the work of Gerbner and Palmer \cite{GePa}.

It is well known that hypergraph Tur\'{a}n problems are considerably more difficult than in the graph setting. In 1965, Erd\H{o}s \cite{Er1} proposed the following conjecture for matchings. Let $M^r_{s}$ denote a matching of size $s$ in an $r$-graph.

\begin{con}[Erd\H{o}s Matching Conjecture~\cite{Er1}]\label{con1}
Let integers $s,r\geq2$ and $n\geq sr-1$. Then $$\ex_r(n,M_{s}^r)\leq \max\left\{ \binom{sr-1}{r}, \binom{n}{r} - \binom{n-s+1}{r} \right\}.$$
\end{con}

The above conjecture is known to hold for the cases $r=2$ \cite{EGa} and $r=3$ \cite{Fran2017,LuMi}. For the general case, Erd\H{o}s \cite{Er1} verified Conjecture \ref{con1} for large $n$. Later, the threshold was improved in several papers by Bollob\'as, Daykin, Erd\H{o}s \cite{68}, Huang, Loh, Sudakov \cite{312}, Frankl, \L uczak, Mieczkowska \cite{215}. For further developments and relevant results concerning this conjecture, one can refer to \cite{Fr1,FrKu,213,KoKu}.
Very recently, Kupavskii and Sokolov \cite{KuSo} completely resolved the conjecture for the case $n\leq 3(s+1)$. Here, we state a recent result due to Frankl \cite{Fr1}. Let $\binom{[n]}{r}-\binom{[s,n]}{r}$ be the $r$-graph on $[n]$ consisting of all hyperedges intersecting the set $[s-1]$.

\begin{thm}[Frankl \cite{Fr1}]\label{Frankl}
Let $r\geq1$, $s\geq2$ and $n\geq (2s-1)r-s+1$. Then $\ex_r(n,M_{s}^r)\leq \binom{n}{r} - \binom{n-s+1}{r}$. Equality holds only for families isomorphic to $\binom{[n]}{r}-\binom{[s,n]}{r}$.
\end{thm}

For $t\geq r$, let $K_t^{(r)}$ denote the complete $r$-graph on $t$ vertices. In 2020, Liu and Wang \cite{LiWa} determined the maximum number of copies of $K_t^{(r)}$ in an $n$-vertex $M^r_{s}$-free $r$-graph for all $t,r,s$ and sufficiently large $n$. It is natural to determine the maximum number of copies of $K_t^{(r)}$ in an $n$-vertex $\cF$-free $r$-graph, as well as the maximum number of other subhypergraphs in an $n$-vertex $M^r_{s}$-free $r$-graph. For the former problem, Zhou, Zhao and Yuan \cite{ZhZY} systematically studied the extremal problem of maximizing the number of copies of $K_t^{(r)}$ in an $n$-vertex expansion-free $r$-graph (for a graph $F$, the \textit{$r$-expansion} $F^r$ of $F$ is the $r$-graph obtained from $F$ by inserting $r-2$ new distinct vertices in each edge of $F$ \cite{Mu}), and also provided a brief survey of generalized Tur\'{a}n problems on hypergraphs.
For the latter problem, very recently Wang and Peng \cite{WaP} determined the maximum number of sunflowers in an $n$-vertex $r$-graph with bounded matching number. Before stating this result, we introduce several notations.

Let $\nu(\cH)$ denote the \textit{matching number} of an $r$-graph $\cH$, i.e., the maximum size of a matching in $\cH$. Given $r$-graphs $\cH$ and $\cG$, let $N(\cH,\cG)$ denote the number of copies of $\cH$ in $\cG$. Given integers $r>t\geq1$ and $k \ge 1$, let $S_{t,k}^r$ denote the $r$-graph with hyperedges $\{e_1, \dots, e_k\}$ such that there exists a $t$-set $T$ with $e_i \cap e_j = T$ for $1\le i < j \le k$. Note that $S_{t,k}^r$ is called an \textit{$r$-uniform sunflower} with $k$ petals and kernel of size $t$. Here, we refer to the set $T$ as the \emph{core} of $S_{t,k}^r$, and all other vertices are referred to as \emph{petal vertices}. 

\begin{thm}[Wang and Peng \cite{WaP}]
Let $s, r \geq 2$, $n \geq 2sr^3$ and $\mathcal{H} \subseteq \binom{[n]}{r}$ with $\nu(\mathcal{H})<s$. Then
\begin{eqnarray*}
N\left(S_{r-1,2}^r, \mathcal{H}\right) \leq N\left(S_{r-1,2}^r, \binom{[n]}{r}-\binom{[s,n]}{r}\right).
\end{eqnarray*}
Moreover, the equality holds if and only if $\mathcal{H} = \binom{[n]}{r}-\binom{[n] \setminus S}{r}$ for some $S \in \binom{[n]}{s-1}$.
\end{thm}

Motivated by the above work, we consider the problem of maximizing the number of copies of the sunflower $S_{r-1,k}^r$ in an $n$-vertex $r$-graph with bounded matching number. We determine the maximum number in $M^r_{s}$-free $r$-graphs on $n$ vertices, which generalizes the above lemma to arbitrary $k\ge 1$. More precisely, we show that the hypergraph $\binom{[n]}{r}-\binom{[s,n]}{r}$ is optimal for all $k\geq1$ and sufficiently large $n$, and characterize all extremal hypergraphs. An interesting phenomenon is that the extremal numbers and extremal hypergraphs are exactly the same for all $k\ge 1$.

\begin{thm}\label{thm1}
Let integers $k,s\geq 1$ and $r\geq2$. Suppose $\cH$ is an $r$-graph on $[n]$ with $\nu(\cH)<s$. Then there exists $n_0(k,s,r)$ such that for all $n >n_0(k,s,r)$,
\begin{eqnarray*}
N(S_{r-1,k}^r, \mathcal{H}) \leq N\left(S_{r-1,k}^r, \binom{[n]}{r} - \binom{[s,n]}{r}\right).
\end{eqnarray*}
Moreover, the equality holds if and only if $\cH$ is the $r$-graph consisting of all hyperedges intersecting a fixed $(s-1)$-set.
\end{thm}

The proof of Theorem~\ref{thm1} will be presented in Section \ref{sec:proofs}. Our main tool is the shifting method, which is introduced in Section \ref{sec:shifting}. Compared with the special case $k=2$ established by Wang and Peng \cite{WaP}, the main difficulty lies in showing that the shifting operation does not decrease the number of copies of $S_{r-1,k}^r$ in any $r$-graph. To resolve this difficulty, we construct an injection via a detailed structural analysis of subhypergraphs (see Lemma \ref{newlem}), thereby establishing the desired monotonicity property. Furthermore, we present a counting method for estimating the number of copies of $S_{r-1,k}^r$ in arbitrary $r$-graphs (see Lemma \ref{lem3.1}). 

We remark that counting the number of copies of $S_{r-1,k}^r$ in $r$-graphs is also related to Tur\'{a}n problems in the $(r-1,k)$-norm, which was introduced by Chen et al. \cite{Chen}. Given integers $r > t \ge 1$ and a real number $k > 0$, the \textit{$(t,k)$-norm} $\|\mathcal{H}\|_{t,k}$ of an $r$-graph $\mathcal{H}$ is the sum of the $k$-th power of the degrees $d_{\mathcal{H}}(T)$ (i.e., the number of hyperedges containing $T$) over all $t$-subsets $T \subset V(\mathcal{H})$. Note that $(r-1,2)$-norm is also called the \textit{$\ell_2$-norm} in the literature. In \cite{Chen}, Chen et al. systematically investigate non-degenerate Tur\'{a}n problems in the $(r-1,k)$-norm.

For non-negative integers $a$ and $b$, let $\binom{a}{b}$ denote the binomial coefficient. In particular, we define $\binom{a}{b}=0$ if $a<b$, and $\binom{0}{0}=1$. By the \emph{Newton expansion of powers} \cite{GrKP}, we have
\begin{eqnarray*}
(d_{\mathcal{H}}(T))^k=\sum_{m=1}^k m! \, S(k,m) \binom{d_{\mathcal{H}}(T)}{m},
\end{eqnarray*}
where $S(k,m)$ denotes the \textit{Stirling numbers of the second kind}.
For integers $k\ge m\ge 1$, $S(k,m)$ is a positive integer counting the number of ways to partition a set of $k$ elements into $m$ non-empty disjoint subsets, with explicit formula
\begin{eqnarray*}
S(k,m)=\frac{1}{m!}\sum_{i=0}^m (-1)^{m-i}\binom{m}{i}i^k.
\end{eqnarray*}
Moreover, we have $S(k,1)=1$ and $S(k,k)=1$. 


Observe that $\sum_{S\in \binom{V(\cH)}{r-1}}\binom{d_{\mathcal{H}}(S)}{k}=N(S_{r-1,k}^r,\cH)$ for $k\geq2$ and $\sum_{S\in \binom{V(\cH)}{r-1}}\binom{d_{\mathcal{H}}(S)}{1}=r\cdot e(\cH)$.
Thus, for all $k\ge 2$, we have
\begin{eqnarray}
\|\mathcal{H}\|_{r-1,k}&=& \sum_{S\in \binom{V(\cH)}{r-1}}(d_{\mathcal{H}}(S))^k \notag\\
&=& \sum_{S\in \binom{V(\cH)}{r-1}}\left(\sum_{m=1}^k m! \, S(k,m) \binom{d_{\mathcal{H}}(S)}{m}\right) \notag\\
&=& \sum_{S\in \binom{V(\cH)}{r-1}}\left(S(k,1)\binom{d_{\mathcal{H}}(S)}{1}+ 2S(k,2)\binom{d_{\mathcal{H}}(S)}{2}+\cdots+ k!S(k,k)\binom{d_{\mathcal{H}}(S)}{k}\right) \notag\\
&=& \left(S(k,1)\sum_{S\in \binom{V(\cH)}{r-1}}\binom{d_{\mathcal{H}}(S)}{1}+ 2S(k,2)\sum_{S\in \binom{V(\cH)}{r-1}}\binom{d_{\mathcal{H}}(S)}{2}+\cdots+ k!S(k,k)\sum_{S\in \binom{V(\cH)}{r-1}}\binom{d_{\mathcal{H}}(S)}{k}\right) \notag\\
&=& r\cdot e(\cH)+ \sum_{i=2}^{k-1}\left(i!S(k,i)\cdot N(S_{r-1,i}^r,\cH)\right)+k!\cdot N(S_{r-1,k}^r,\cH).
\end{eqnarray}
This implies that determining the $(r-1,k)$-norm $\|\mathcal{H}\|_{r-1,k}$ of an $r$-graph $\mathcal{H}$ is equivalent to counting copies of $S_{r-1,i}^r$ in $\mathcal{H}$ for every $1\le i\leq k$.

Let $\ex_{t,k}(n,\mathcal{F})$ denote the maximum of $\|\mathcal{H}\|_{t,k}$ over all $n$-vertex $\cF$-free $r$-graphs $\cH$. This notation is an extension of the Tur\'an number $\ex_r(n,\cF)$, since $\ex_r(n,\cF)=\ex_{t,1}(n,\cF)/\binom{r}{t}$. Recall that Theorem~\ref{thm1} shows that the extremal numbers and extremal hypergraphs of $S_{r-1,i}^r$ are exactly the same for all $2\leq i\leq k$. Thus, combining (1.1), Theorems \ref{Frankl} and \ref{thm1}, we obtain the following result, which confirms that Erd\H{o}s Matching Conjecture in the $(r-1,k)$-norm holds.

\begin{thm}[\bf Erd\H{o}s Matching Conjecture in $(r-1,k)$-norm]\label{thm2}
Let integers $s\geq 1$ and $k,r\geq2$. Then there exists $n_0(k,s,r)$ such that for all $n >n_0(k,s,r)$,
\begin{eqnarray*}
\ex_{r-1,k}(n,M_s^r)\leq \left\|\binom{[n]}{r} - \binom{[s,n]}{r}\right\|_{r-1,k}.
\end{eqnarray*}
Moreover, the equality holds if and only if $\cH$ is the $r$-graph consisting of all hyperedges intersecting a fixed $(s-1)$-set.
\end{thm}

The special case $k=2$ of Theorem \ref{thm2} was previously established by Brooks and Linz \cite{BrL}. It is worth noting that when $s=2$, Theorem \ref{thm2} yields a version of the Erd\H{o}s--Ko--Rado theorem in the $(r-1,k)$-norm. A family $\cF$ is called \textit{$t$-intersecting} if $|A\cap B| \geq t$ for all $A, B \in \mathcal{F}$. When $t=1$, we say the family is \textit{intersecting}. Recall that the classic Erd\H{o}s--Ko--Rado theorem states that if $\cF\subseteq \binom{[n]}{r}$ is an intersecting family with $n>2r$, then $|\mathcal{F}| \leq \binom{n-1}{r-1}$, with equality only for families isomorphic to $\binom{[n]}{r}-\binom{[2,n]}{r}$.

\begin{thm}[\bf Erd\H{o}s--Ko--Rado Theorem in $(r-1,k)$-norm]\label{cor1}
Let $k,r\geq2$ be integers and $\mathcal{F} \subseteq \binom{[n]}{r}$ be an intersecting family. Then there exists $n_0(k,r)$ such that for all $n >n_0(k,r)$,
\begin{eqnarray*}
\|\mathcal{F}\|_{r-1,k}\leq \binom{n-1}{r-1}\left(1+(r-1)(n-r+1)^{k-1}\right).
\end{eqnarray*}
Moreover, the equality holds only for families isomorphic to $\binom{[n]}{r}-\binom{[2,n]}{r}$.
\end{thm}

\section{\normalsize The shifting operation}\label{sec:shifting}
In this section, we introduce the shifting operation, and establish several fundamental properties of this operation. The shifting method was originally introduced by Erd\H{o}s--Ko--Rado \cite{EKR} and further developed by Frankl \cite{Fr111}.

\subsection{\normalsize Definitions and basic properties}\label{sec:2.1}
Let $\mathcal{H}$ be an $r$-graph on the vertex set $[n]$. For integers $i,j$ with $1 \leq i < j \leq n$ and any $e \in E(\mathcal{H})$, the \textit{shifting operator} $S_{ij}$ is defined by
\begin{eqnarray*}
S_{ij}(e) =
\begin{cases}
(e \setminus \{j\}) \cup \{i\}, & \text{if } j \in e, i \notin e \text{ and } (e \setminus \{j\}) \cup \{i\} \notin E(\mathcal{H}); \\[6pt]
e, & \text{otherwise}.
\end{cases}
\end{eqnarray*}
Set
\begin{eqnarray*}
S_{ij}(\mathcal{H}) = \{ S_{ij}(e) : e \in E(\mathcal{H})\}.
\end{eqnarray*}

Clearly, $S_{ij}$ is a bijection from $\mathcal{H}$ to $S_{ij}(\mathcal{H})$. An $r$-graph $\mathcal{H}$ is called \textit{shifted} (or \textit{stable}) if $S_{ij}(\mathcal{H}) = \mathcal{H}$ for all $1 \leq i < j \leq n$. Frankl \cite{Fr111} showed that any $r$-graph $\mathcal{H}$ can be transformed into a shifted $r$-graph by applying the shifting operation iteratively, and that the shifting operation does not increase the matching number of an $r$-graph.

\begin{lem}[Frankl \cite{Fr111}]\label{lem2.1}
Let $s \geq 2$ be an integer. Let $\mathcal{H}$ be an $r$-uniform graph on $[n]$ with $\nu(\mathcal{H}) < s$. Let $1 \leq i < j \leq n$ be a pair of vertices, then $\nu(S_{ij}(\mathcal{H})) < s$.
\end{lem}

\begin{figure}[h]
  \centering
  \includegraphics[scale=0.6]{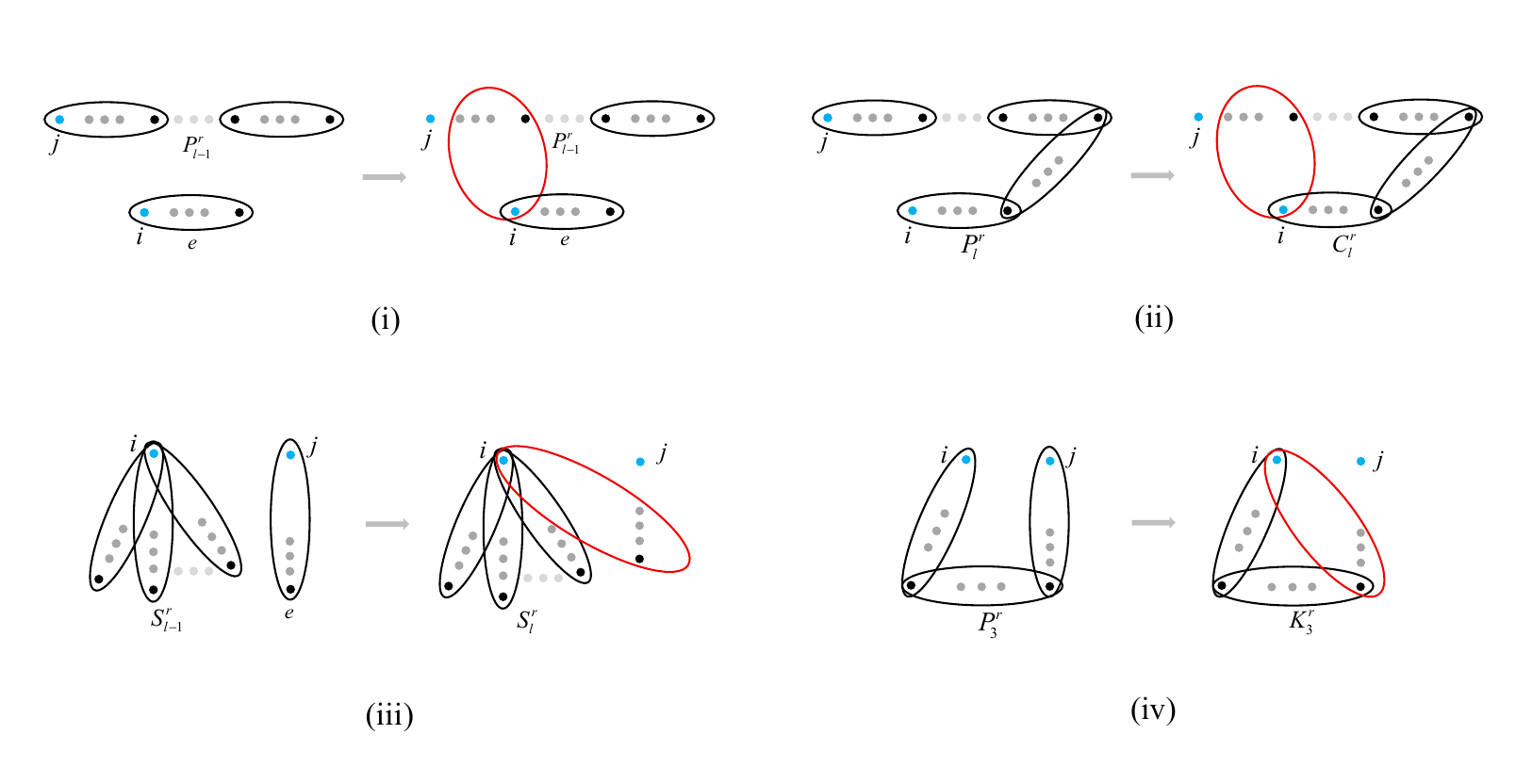}
  \caption{\small{Shifting operations for several expansions}}\label{fig1}
\end{figure}

The above lemma implies that the shifting operation preserves the property that the matching number is less than $s$, i.e., $S_{ij}(\mathcal{H})$ is still $M^r_s$-free. Note that $M_s^r$ is the $r$-uniform expansion of $M_s$. However, such freeness properties are not preserved under shifting in general. Let $P_\ell$, $C_\ell$ and $S_\ell$ denote a path, a cycle and a star with $\ell$ edges, respectively. Denote by $K_\ell$ the complete graph on $\ell$ vertices. The examples in Figure \ref{fig1} demonstrate that the shifting operation does not preserve $P^r_\ell$-freeness (see {(i)}), $C^r_\ell$-freeness (see {(ii)}), $S^r_\ell$-freeness (see {(iii)}), or $K^r_3$-freeness (see {(iv)}).

\subsection{\normalsize Counting subhypergraphs under shifting}\label{sec:2.2}
In this subsection, we investigate how the shifting operation affects the number of copies of subhypergraphs in an $r$-graph. This problem is crucial, as it determines whether we may first apply the shifting operation to $\cH$ when counting copies of subhypergraphs in $\mathcal{H}$. Liu and Wang \cite{LiWa} showed that the shifting operation does not decrease the number of copies of $K_s^{(r)}$ in any $r$-graph.

\begin{lem}[Liu and Wang \cite{LiWa}]\label{lem2.2}
Let $\mathcal{H}$ be an $r$-graph on $[n]$. For any $i,j \in [n]$ with $i < j$ and $s \geq r$, we have $N(K_s^{(r)},S_{ij}(\mathcal{H}))\geq N(K_s^{(r)},\cH)$.
\end{lem}

Wang and Peng \cite{WaP} showed that the shifting operation does not decrease the number of copies of $S^r_{r-1,2}$ in any $r$-graph.

\begin{lem}[Wang and Peng \cite{WaP}]\label{lem2.3}
Let $\mathcal{H}$ be an $r$-graph on $[n]$ and $1\leq i<j\leq n$. Then $N(S^r_{r-1,2},S_{ij}(\mathcal{H}))\geq N(S^r_{r-1,2},\cH)$.
\end{lem}

Wang and Peng raised the natural question: for which subhypergraphs $\cF$, apart from $K_s^{(r)}$ and $S^r_{r-1,2}$, does the shifting operation not decrease (or increase) the number of copies of $\cF$ in any $r$-graph? Here, we give an answer to this question for the sunflower $S^r_{r-1,k}$ for all $k\geq1$. More precisely, we establish that the shifting operation does not decrease the number of copies of $S^r_{r-1,k}$ in any $r$-graph, which extends Lemma \ref{lem2.3}.

\begin{lem}\label{newlem}
Let $k\geq 1$ and $\cH$ be an $r$-graph on $[n]$. For any $i,j \in [n]$ with $i < j$, we have $N(S^{r}_{r-1,k},S_{ij}(\cH))\geq N(S^{r}_{r-1,k},\cH)$.
\end{lem}
\begin{proof}[\bf Proof]
If $k=1$, then $S^{r}_{r-1,1}$ is a single hyperedge, and thus the conclusion is trivial. Next, suppose $k\geq2$. Recall that for every $e\in E(\cH)$, we have $S_{ij}(e)=(e \setminus \{j\}) \cup \{i\}$ if $j \in e$, $i \notin e$ and $(e \setminus \{j\}) \cup \{i\} \notin E(\mathcal{H})$. Now, we define a map $\phi$ from copies of $S^{r}_{r-1,k}$ in $\cH$ to copies of $S^{r}_{r-1,k}$ in $S_{ij}(\cH)$ as follows. For convenience, let $F=\{e_1,\dots,e_k\}$ denote a copy of $S^{r}_{r-1,k}$ in $\mathcal{H}$. Let $C(F)$ and $L(F)$ denote the core and the set of petal vertices of $F$, respectively.
\begin{itemize}
    \item If $j \notin V(F)$, then set $\phi(F)=S_{ij}(F)$. Note that in this case $S_{ij}(F)=F$.
    \item If $i,j \in C(F)$, then set $\phi(F)=S_{ij}(F)$. Note that in this case $S_{ij}(F)=F$.
    \item Suppose $j \in C(F)$ and $i\in L(F)\cap e_x$ for some $x\in [k]$. If $(e_m\setminus \{j\}) \cup \{i\} \in E(\cH)$ for all $m\in [k]\backslash x$, then set $\phi(F)=S_{ij}(F)$, and in this case $S_{ij}(F)=F$. Otherwise, set
        \begin{eqnarray}
        \phi(F)=e_x\cup\big\{(e_m \setminus \{j\}) \cup \{i\}: m\in [k]\backslash x\big\}.
        \end{eqnarray}
    \item If $j \in L(F)$ and $i\in C(F)$, then set $\phi(F)=S_{ij}(F)$. Note that in this case $S_{ij}(F)=F$.
    \item If $i,j \in L(F)$, then set $\phi(F)=S_{ij}(F)$. Note that in this case $S_{ij}(F)=F$.
    \item Suppose $i\notin V(F)$ and $j \in C(F)$. If $(e_m\setminus \{j\}) \cup \{i\} \in E(\cH)$ for all $m\in [k]$, then set $\phi(F)=S_{ij}(F)$, and in this case $S_{ij}(F)=F$. Otherwise, set \begin{eqnarray}
        \phi(F)=\big\{(e_m \setminus \{j\}) \cup \{i\}: m\in [k]\big\}.
        \end{eqnarray}
    \item Suppose $i\notin V(F)$ and $j \in L(F)\cap e_x$ for some $x\in [k]$. If $(e_x\setminus \{j\}) \cup \{i\} \in E(\cH)$, then set $\phi(F)=S_{ij}(F)$, and in this case $S_{ij}(F)=F$. Otherwise, set \begin{eqnarray}
        \phi(F)=\big\{(e_x\setminus \{j\})\cup\{i\}\big\} \cup \big\{e_m: m\in [k]\backslash x\big\}.
        \end{eqnarray}
\end{itemize}

It is straightforward to check that $\phi(F)$ is a copy of $S^{r}_{r-1,k}$ in $S_{ij}(\cH)$ for every such $F$ in $\cH$ by the definition of $S_{ij}$. It suffices to show that $\phi$ is injective. Suppose $F_1$ and $F_2$ are distinct copies of $S_{r-1,k}^r$ in $\mathcal{H}$ such that $\phi(F_1)=\phi(F_2)$. By the definition of $\phi$, we can assume that $\phi(F_1)\neq F_1$ or $\phi(F_2)\neq F_2$. Otherwise, we have $\phi(F_1)=F_1$ and $\phi(F_2)=F_2$, and thus $F_1=F_2$, a contradiction. Without loss of generality, we assume that $\phi(F_1)\neq F_1$. Thus, $\phi(F_1)$ satisfies condition (2.1), (2.2), or (2.3).

\smallskip
\smallskip
\noindent {\bf{Case 1.}} $\phi(F_1)$ satisfies (2.1) and all corresponding conditions.
\smallskip
\smallskip

It follows that $\phi(F_1)=\phi(F_2)=e_x\cup\big\{(e_m \setminus \{j\}) \cup \{i\}: m\in [k]\backslash x\big\}$. Since $\phi(F_1)$ satisfies (2.1) and all corresponding conditions, by the definition of $\phi$, we have $j \in C(F_1)$, $i\in L(F_1)\cap e_x$ for some $x\in [k]$, and $(e_y\setminus \{j\}) \cup \{i\} \notin E(\cH)$ for some $y\in [k]\backslash x$.

If $\phi(F_2)=F_2$, then $F_2=e_x\cup\big\{(e_m \setminus \{j\}) \cup \{i\}: m\in [k]\backslash x\big\}\subseteq \cH$, which contradicts $(e_y\setminus \{j\}) \cup \{i\} \notin E(\cH)$ for some $y\in [k]\backslash x$.
If $\phi(F_2)$ satisfies (2.1) and all corresponding conditions, then by the definition of $\phi$, we have $F_1=F_2$ as $\phi(F_1)=\phi(F_2)=e_x\cup\big\{(e_m \setminus \{j\}) \cup \{i\}: m\in [k]\backslash x\big\}$. This contradicts the assumption that $F_1\neq F_2$.
If $\phi(F_2)$ satisfies (2.2) and all corresponding conditions, then by the definition of $\phi$, we have $j\notin V(\phi(F_2))$, which contradicts $j\in V(\phi(F_1))=V(\phi(F_2))$.
Similarly, if $\phi(F_2)$ satisfies (2.3) and all corresponding conditions, then we have $j\notin V(\phi(F_2))$, which contradicts $j\in V(\phi(F_1))=V(\phi(F_2))$.

\smallskip
\smallskip
\noindent {\bf{Case 2.}} $\phi(F_1)$ satisfies (2.2) and all corresponding conditions.
\smallskip
\smallskip

It follows that $\phi(F_1)=\phi(F_2)=\big\{(e_m \setminus \{j\}) \cup \{i\}: m\in [k]\big\}$. Since $\phi(F_1)$ satisfies (2.2) and all corresponding conditions, by the definition of $\phi$, we have $i\notin V(F_1)$, $j \in C(F_1)$, and $(e_x\setminus \{j\}) \cup \{i\} \notin E(\cH)$ for some $x\in [k]$.

If $\phi(F_2)=F_2$, then $F_2=\big\{(e_m \setminus \{j\}) \cup \{i\}: m\in [k]\big\}\subseteq \cH$, which contradicts $(e_x\setminus \{j\}) \cup \{i\} \notin E(\cH)$ for some $x\in [k]$.
If $\phi(F_2)$ satisfies (2.1) and all corresponding conditions, then by the same argument as in Case 1, we obtain a contradiction.
If $\phi(F_2)$ satisfies (2.2) and all corresponding conditions, then by the definition of $\phi$, we have $F_1=F_2$ as $\phi(F_1)=\phi(F_2)=\big\{(e_m \setminus \{j\}) \cup \{i\}: m\in [k]\big\}$. This contradicts the assumption that $F_1\neq F_2$.
If $\phi(F_2)$ satisfies (2.3) and all corresponding conditions, then by the definition of $\phi$, we have $i\in L(\phi(F_2))$, which contradicts $i\in C(\phi(F_1))=C(\phi(F_2))$.

\smallskip
\smallskip
\noindent {\bf{Case 3.}} $\phi(F_1)$ satisfies (2.3) and all corresponding conditions.
\smallskip
\smallskip

It follows that $\phi(F_1)=\phi(F_2)=\big\{(e_x\setminus \{j\})\cup\{i\}\big\} \cup \big\{e_m: m\in [k]\backslash x\big\}$. Since $\phi(F_1)$ satisfies (2.3) and all corresponding conditions, by the definition of $\phi$, we have $i\notin V(F_1)$, $j \in L(F_1)\cap e_x$ for some $x\in [k]$, and $(e_x\setminus \{j\}) \cup \{i\} \notin E(\cH)$.

If $\phi(F_2)=F_2$, then $F_2=\big\{(e_x\setminus \{j\})\cup\{i\}\big\} \cup \big\{e_m: m\in [k]\backslash x\big\}\subseteq \cH$, which contradicts $(e_x\setminus \{j\}) \cup \{i\} \notin E(\cH)$.
If $\phi(F_2)$ satisfies (2.1) and all corresponding conditions, then by the same argument as in Case 1, we obtain a contradiction.
If $\phi(F_2)$ satisfies (2.2) and all corresponding conditions, then by the same argument as in Case 2, we obtain a contradiction.
If $\phi(F_2)$ satisfies (2.3) and all corresponding conditions, then by the definition of $\phi$, we have $F_1=F_2$ as $\phi(F_1)=\phi(F_2)=\big\{(e_x\setminus \{j\})\cup\{i\}\big\} \cup \big\{e_m: m\in [k]\backslash x\big\}$. This contradicts the assumption that $F_1\neq F_2$.

Therefore, $\phi(F_1)=\phi(F_2)$ implies $F_1=F_2$. Consequently, $\phi$ is injective. This completes the proof.
\end{proof}

\begin{figure}[h]
  \centering
  \includegraphics[scale=0.55]{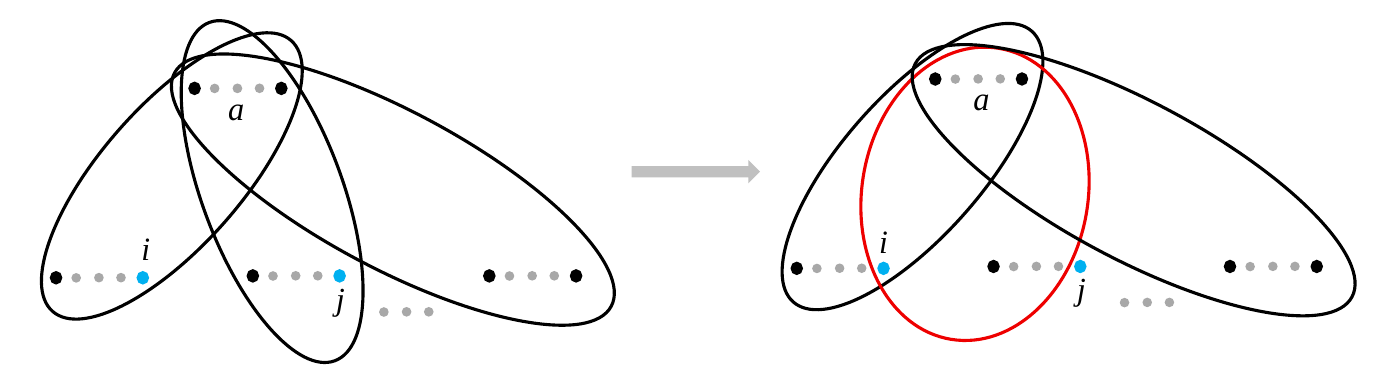}
  \caption{\small{The shifting operation for $S^{r}_{a,k}$}}\label{fig2}
\end{figure}

\begin{rem}
This non-decreasing property of the shifting operation does not hold for all sunflowers. For example, the shifting operation may strictly decrease the number of copies of $S^{r}_{a,k}$ in an $r$-graph whenever $1\leq a\leq r-2$ and $k\geq2$ (see Figure \ref{fig2}).
\end{rem}

\section{\normalsize Proofs}\label{sec:proofs}
In this section, we present a counting method for estimating the number of copies of $S_{r-1,k}^r$ in arbitrary $r$-graphs (see Lemma \ref{lem3.1}). Utilizing this counting method together with the monotonicity properties of the shifting operation, we provide the proof of Theorem~\ref{thm1}. More precisely, our proof strategy is structured as follows. First, using the fundamental properties of the shifting operation, namely the preservation of $M_s^r$-freeness (see Lemma \ref{lem2.1}) and the non-decreasing monotonicity of $S_{r-1,k}^r$ copy counts (see Lemma \ref{newlem}), we reduce the problem to shifted hypergraphs. Second, we use our established counting method for sunflowers to characterize the extremal shifted hypergraphs. Finally, we remove the shifted constraint and derive the extremal structure for general $r$-graphs.

\subsection{\normalsize Counting copies of $S^r_{r-1,k}$ in $r$-graphs}
In this subsection, we provide a counting method for estimating the number of copies of $S_{r-1,k}^r$ in any $r$-graph $\cH$ for all $k\ge1$.

Let $\cH$ be an $r$-graph and $U\subseteq V(\cH)$ be a nonempty subset. Let $\cH-U$ denote the $r$-graph obtained from $\cH$ by removing the vertices in $U$ and the hyperedges incident to them. 
The \textit{link hypergraph} of a vertex $v$ of $\cH$ is an $(r-1)$-graph, defined as
\begin{eqnarray*}
\cL_v(\cH)=\big\{e\backslash\{v\}\,:\,e\in E(\cH), v\in e\big\}.
\end{eqnarray*}

\begin{lem}\label{lem3.1}
Let $\cH$ be an $r$-graph on $[n]$, $i\in [n]$ and $k\geq1$.

\textbf{(i)} If $k=1$, then
\begin{eqnarray*}
N(S_{r-1,1}^r, \cH)=N(S_{r-1,1}^r, \cH-\{i\})+ N(S_{r-2,1}^{r-1}, \cL_i(\cH)).
\end{eqnarray*}

\textbf{(ii)} If $k=2$, then
\begin{eqnarray*}
N(S_{r-1,2}^r, \cH)\leq N(S_{r-1,2}^r, \cH-\{i\})+ N(S_{r-2,2}^{r-1}, \cL_i(\cH))+ rN(S_{r-1,1}^r, \cH-\{i\}),
\end{eqnarray*}

\textbf{(iii)} If $k\geq3$, then 
\begin{eqnarray*}
N(S_{r-1,k}^r, \cH)\leq N(S_{r-1,k}^r, \cH-\{i\})+ N(S_{r-2,k}^{r-1}, \cL_i(\cH))+ N(S_{r-1,k-1}^r, \cH-\{i\}),
\end{eqnarray*}

\textbf{(iv)} If $k\geq2$, then
\begin{eqnarray*}
N(S_{r-1,k}^r, \cH)\leq N(S_{r-1,k}^r, \cH-\{i\})+ N(S_{r-2,k}^{r-1}, \cL_i(\cH))+ e(\cH-\{i\})\cdot \frac{r}{k-1}\binom{n-r-1}{k-2}.
\end{eqnarray*}
\end{lem}

\begin{proof}[\bf Proof]
Observe that for $k=1$, we have $N(S_{r-1,1}^r, \cH)=e(\cH)$. Thus,
$$N(S_{r-1,1}^r, \cH)=e(\cH)=e(\cH-\{i\})+e(\cL_i(\cH))=N(S_{r-1,1}^r, \cH-\{i\})+ N(S_{r-2,1}^{r-1}, \cL_i(\cH)),$$
which establishes \textbf{(i)}.

Now suppose $k\geq2$. We partition all copies of $S_{r-1,k}^r$ in $\cH$ into three disjoint classes:
\begin{itemize}
    \item those not containing the vertex $i$, whose number is denoted by $\bar{N}_i(S_{r-1,k}^r, \cH)$.
    \item those in which $i$ is the core vertex, whose number is denoted by $N^1_i(S_{r-1,k}^r, \cH)$.
    \item those in which $i$ is a petal vertex, whose number is denoted by $N^2_i(S_{r-1,k}^r, \cH)$.
\end{itemize}
It is straightforward to verify that $\bar{N}_i(S_{r-1,k}^r, \cH)=N(S_{r-1,k}^r, \cH-\{i\})$ and $N^1_i(S_{r-1,k}^r, \cH)=N(S_{r-2,k}^{r-1}, \cL_i(\cH))$. Observe that $N^2_i(S_{r-1,2}^r, \cH)\leq r\cdot e(\cH-\{i\})=rN(S_{r-1,1}^r, \cH-\{i\})$ and $N^2_i(S_{r-1,k}^r, \cH)\leq N(S_{r-1,k-1}^r, \cH-\{i\})$ for $k\geq3$, which establishes \textbf{(ii)} and \textbf{(iii)}. Next, we use a different method to estimate $N^2_i(S_{r-1,k}^r, \cH)$ for $k\geq2$. We can choose a copy $S_{r-1,k}^r$ with petal vertex $i$ by the following three steps: first pick a hyperedge from $E(\cH-\{i\})$, then select its core from this hyperedge, and finally choose $k-2$ vertices from the remaining vertices of $V(\cH-\{i\})$ to serve as petal vertices. Note that in this way, each copy of $S_{r-1,k}^r$ with petal vertex $i$ is chosen exactly $k-1$ times. Therefore,
\begin{eqnarray*}
N^2_i(S_{r-1,k}^r, \cH)\leq e(\cH-\{i\})\cdot \binom{r}{r-1}\binom{n-r-1}{k-2}\Big/(k-1)=e(\cH-\{i\})\cdot \frac{r}{k-1}\binom{n-r-1}{k-2}.
\end{eqnarray*}
Combining these three estimates, we obtain the desired result in \textbf{(iv)}. This completes the proof.
\end{proof}

\subsection{\normalsize Shifted extremal hypergraphs}
As established earlier, it follows from Lemmas \ref{lem2.1} and \ref{newlem} that the shifting operation preserves $M_s$-freeness and does not decreas the number of copies of $S_{r-1,k}^r$ in any $r$-graph. Therefore, we may assume that $\cH$ is shifted. Next, we prove Theorem \ref{thm1} under the assumption that $\cH$ is shifted.


\begin{thm}\label{thm1.1}
Let integers $s\geq 1$ and $k,r\geq2$. Suppose $\cH$ is a shifted $r$-graph on $[n]$ with $\nu(\cH)< s$. Then there exists $n_0(k,s,r)$ such that for all $n >n_0(k,s,r)$,
\begin{eqnarray*}
N(S_{r-1,k}^r, \mathcal{H}) \leq N\left(S_{r-1,k}^r, \binom{[n]}{r} - \binom{[s,n]}{r}\right).
\end{eqnarray*}
Moreover, equality holds if and only if $\mathcal{H} = \binom{[n]}{r} - \binom{[s,n]}{r}$.
\end{thm}

\begin{proof}[\bf Proof]
We proceed by induction on $s$. If $s=1$, then the conclusion trivially holds. Now we assume that $s>1$ and the conclusion holds for all $s'<s$. Suppose $\cH$ is a shifted $r$-graph with $\nu(\cH)< s$ that maximizes $N(S_{r-1,k}^r,\mathcal{H})$. Since $\binom{[n]}{r}-\binom{[s,n]}{r}$ is shifted with $\nu(\mathcal{H})< s$, we have
\begin{eqnarray}
N(S_{r-1,k}^r, \mathcal{H}) \geq N\left(S_{r-1,k}^r, \binom{[n]}{r} - \binom{[s,n]}{r}\right).
\end{eqnarray}
We may assume that $\cH$ is edge-maximal (i.e., adding any new hyperedge to $E(\mathcal{H})$ would yield a copy of $M_s^{r}$). Otherwise, we may add an additional hyperedge to $E(\mathcal{H})$, which does not decrease the number of copies of $S_{r-1,k}^r$. For convenience, set $\cH^*:=\binom{[n]}{r} - \binom{[s,n]}{r}$. We have the following claim.

\begin{cla}\label{claim1}
There exists $n_1(k,s,r)$ such that for all $n >n_1(k,s,r)$, we have $\{1,(s-1)r+2,\dots,sr\}\in E(\cH)$.
\end{cla}
\begin{proof}[\bf Proof of Claim]
Suppose to the contrary that $\{1,(s-1)r+2,\dots,sr\}\notin E(\cH)$. Since $\cH$ is shifted, $|e\cap[(s-1)r+2,n]|\leq r-2$ for any $e\in E(\cH)$. This implies that $\cH\subseteq \cH'$, where $\cH'$ is the $r$-graph consisting of all hyperedges intersecting $[(s-1)r+1]$ in at least two vertices. Therefore, $N(S_{r-1,k}^r, \mathcal{H})\leq N(S_{r-1,k}^r, \cH')$. Next, we prove that $N(S_{r-1,k}^r, \cH')< N(S_{r-1,k}^r, \cH^*)$.

For any $r$-graph $\cG$, we may count the number of copies of $S_{r-1,k}^r$ by first choosing an $(r-1)$-set from $V(\cG)$, and then selecting $k$ distinct hyperedges from all hyperedges of $\mathcal{G}$ containing it. Note that there are $r-1$ types of $(r-1)$-sets in $\cH^*$, namely those intersecting $[s-1]$ in at least one and at most $r-1$ vertices. Thus,
\begin{eqnarray}
N(S_{r-1,k}^r, \cH^*)&=& \binom{s-1}{r-1}\cdot\binom{n-r+1}{k}+\binom{s-1}{r-2}\binom{n-s+1}{1}\cdot\binom{n-r+1}{k} \notag\\
&{}& +\cdots+\binom{s-1}{1}\binom{n-s+1}{r-2}\cdot\binom{n-r+1}{k}+\binom{n-s+1}{r-1}\cdot\binom{s-1}{k} \notag\\
&=& \binom{s-1}{1}\binom{n-s+1}{r-2}\binom{n-r+1}{k}+O(n^{r+k-3}) \notag\\
&=& \frac{(s-1)n^{r+k-2}}{(r-2)!k!}+O(n^{r+k-3}).
\end{eqnarray}

Similarly, there are $r-2$ types of $(r-1)$-sets in $\cH'$, namely those intersecting $[(s-1)r+1]$ in at least $2$ and at most $r-1$ vertices. Thus,
\begin{eqnarray}
N\left(S_{r-1,k}^r, \cH'\right)&=& \binom{(s-1)r+1}{r-1}\cdot\binom{n-r+1}{k}+\binom{(s-1)r+1}{r-2}\binom{n-(s-1)r-1}{1}\cdot\binom{n-r+1}{k} \notag\\
&{}& +\cdots+\binom{(s-1)r+1}{2}\binom{n-(s-1)r-1}{r-3}\cdot\binom{n-r+1}{k} \notag\\
&{}& +\binom{(s-1)r+1}{1}\binom{n-(s-1)r-1}{r-2}\cdot\binom{(s-1)r}{k} \notag\\
&=& \binom{(s-1)r+1}{2}\binom{n-(s-1)r-1}{r-3}\binom{n-r+1}{k}+O(n^{r+k-4}) \notag\\
&=& \frac{\binom{(s-1)r+1}{2}n^{r+k-3}}{(r-3)!k!}+O(n^{r+k-4}).
\end{eqnarray}

Combining (3.2) and (3.3), it follows that there exists $n_1(k,s,r)$ such that for all $n >n_1(k,s,r)$, we have $N(S_{r-1,k}^r, \cH')< N(S_{r-1,k}^r, \cH^*)$. Therefore, $N(S_{r-1,k}^r, \mathcal{H})<N(S_{r-1,k}^r, \cH^*)$, which contradicts inequality (3.1).
\end{proof}

We also claim that $\{1,n-r+2,\dots,n\}\in E(\cH)$. Otherwise, suppose $\{1,n-r+2,\dots,n\}\notin E(\cH)$. Since $\cH$ is edge-maximal, $\cH\cup\{1,n-r+2,\dots,n\}$ contains a copy of $M_s^r$, where $\{1,n-r+2,\dots,n\}\in M_s^r$. This implies that $\cH$ contains a copy of $M_{s-1}^r$ such that $1,n-r+2,\dots,n\notin V(M_{s-1}^r)$. Since $\cH$ is shifted, there exists a copy of $M_{s-1}^r$ in $[2,(s-1)r+1]$, denoted by $M$. By Claim \ref{claim1}, $\{1,(s-1)r+2,\dots,sr\}\in E(\cH)$. Hence, $M\cup\{1,(s-1)r+2,\dots,sr\}$ forms a copy of $M_s^r$ in $\cH$, a contradiction. Thus, $\{1,n-r+2,\dots,n\}\in E(\cH)$. Then $\cL_1(\cH)=\binom{[2,n]}{r-1}$ as $\cH$ is shifted. This implies that $\nu(\cH-\{1\})<s-1$. Indeed, if $\cH-\{1\}$ contains a copy of $M_{s-1}^r$ (denoted by $M'$), then we may find a copy of $M_{s}^r:=M'\cup\big\{ \{1\}\cup S\big\}$ in $\cH$ for some $S\in \binom{[2,n]}{r-1}$, a contradiction.

Using the counting method from Claim~\ref{claim1}, we obtain that for all $n >n_1(k,s,r)$,
\begin{eqnarray}
N(S_{r-2,k}^{r-1}, \cL_1(\cH))= N\left(S_{r-2,k}^{r-1}, \binom{[2,n]}{r-1}\right)= \binom{n-1}{r-2}\cdot\binom{n-r+1}{k}.
\end{eqnarray}

Since $\nu(\cH-\{1\})<s-1$, by Theorem \ref{Frankl}, we get 
\begin{eqnarray}
N(S_{r-1,1}^r, \cH-\{1\})=e(\cH-\{1\})\leq N\left(S_{r-1,1}^r, \binom{[2,n]}{r}-\binom{[s,n]}{r}\right) 
\end{eqnarray}
for all $n>\max\{(2s-1)r-s,n_1(k,s,r)\}$. Meanwhile, by the induction hypothesis, 
\begin{eqnarray}
N(S_{r-1,k-1}^r, \cH-\{1\})\leq N\left(S_{r-1,k-1}^r, \binom{[2,n]}{r}-\binom{[s,n]}{r}\right)
\end{eqnarray}
for $k\geq3$. 

Combining (3.4)--(3.6), and by Lemma \ref{lem3.1}\,\textbf{(ii)} and \textbf{(iii)}, we have
\begin{eqnarray*}
N(S_{r-1,2}^r, \cH)&\leq& N(S_{r-1,2}^r, \cH-\{1\})+ N(S_{r-2,2}^{r-1}, \cL_1(\cH))+ rN(S_{r-1,1}^r, \cH-\{1\}) \\
&\leq& N(S_{r-1,2}^r, \cH-\{1\})+\binom{n-1}{r-2}\binom{n-r+1}{2}+ rN\left(S_{r-1,1}^r, \binom{[2,n]}{r}-\binom{[s,n]}{r}\right),
\end{eqnarray*}
and
\begin{eqnarray*}
N(S_{r-1,k}^r, \cH)&\leq& N(S_{r-1,k}^r, \cH-\{1\})+ N(S_{r-2,k}^{r-1}, \cL_1(\cH))+ N(S_{r-1,k-1}^r, \cH-\{1\}) \\
&\leq& N(S_{r-1,k}^r, \cH-\{1\})+\binom{n-1}{r-2}\binom{n-r+1}{k}+ N\left(S_{r-1,k-1}^r, \binom{[2,n]}{r}-\binom{[s,n]}{r}\right)
\end{eqnarray*}
for $k\geq3$. Therefore, 
\begin{eqnarray}
N(S_{r-1,2}^r, \cH^*)&\leq& N(S_{r-1,2}^r, \cH-\{1\})+\binom{n-1}{r-2}\binom{n-r+1}{2} \notag\\
&{}& +rN\left(S_{r-1,1}^r, \binom{[2,n]}{r}-\binom{[s,n]}{r}\right),
\end{eqnarray}
and
\begin{eqnarray}
N(S_{r-1,k}^r, \cH^*)&\leq& N(S_{r-1,k}^r, \cH-\{1\})+\binom{n-1}{r-2}\binom{n-r+1}{k} \notag\\
&{}& +N\left(S_{r-1,k-1}^r, \binom{[2,n]}{r}-\binom{[s,n]}{r}\right)
\end{eqnarray}
for $k\geq3$. 

To estimate $N(S_{r-1,k}^r, \cH-\{1\})$, let us use the method from Lemma~\ref{lem3.1} to count the number of copies of $S_{r-1,k}^r$ in $\cH^*$. We partition all copies of $S_{r-1,k}^r$ in $\cH^*$ into three disjoint classes:
\begin{itemize}
    \item those not containing the vertex $1$, whose number is denoted by $\bar{N}_1(S_{r-1,k}^r, \cH^*)$.
    \item those in which $1$ is the core vertex, whose number is denoted by $N^1_1(S_{r-1,k}^r, \cH^*)$.
    \item those in which $1$ is a petal vertex, whose number is denoted by $N^2_1(S_{r-1,k}^r, \cH^*)$.
\end{itemize}
It is straightforward to verify that
\begin{eqnarray}
\bar{N}_1(S_{r-1,k}^r, \cH^*)=N(S_{r-1,k}^r, \cH^*-\{1\})=N\left(S_{r-1,k}^r, \binom{[2,n]}{r}-\binom{[s,n]}{r}\right),
\end{eqnarray}
and
\begin{eqnarray}
N^1_1(S_{r-1,k}^r, \cH^*)=N(S_{r-2,k}^{r-1}, \cL_1(\cH^*))=N\left(S_{r-2,k}^{r-1}, \binom{[2,n]}{r-1}\right)= \binom{n-1}{r-2}\binom{n-r+1}{k}.
\end{eqnarray}
Now we estimate $N^2_1(S_{r-1,k}^r, \cH^*)$. Since $\cL_1(\cH^*)=\binom{[2,n]}{r-1}$, we have 
\begin{eqnarray}
N^2_1(S_{r-1,2}^r, \cH^*)= rN(S_{r-1,1}^r, \cH^*-\{1\})= rN\left(S_{r-1,1}^r, \binom{[2,n]}{r}-\binom{[s,n]}{r}\right), 
\end{eqnarray}
and
\begin{eqnarray}
N^2_1(S_{r-1,k}^r, \cH^*)= N(S_{r-1,k-1}^r, \cH^*-\{1\})= N\left(S_{r-1,k-1}^r, \binom{[2,n]}{r}-\binom{[s,n]}{r}\right)
\end{eqnarray}
for $k\geq3$. 

Combining (3.7)--(3.12), we obtain that for all $n>\max\{(2s-1)r-s,n_1(k,s,r)\}$ and $k\geq2$, 
\begin{eqnarray*}
N(S_{r-1,k}^r, \cH-\{1\})\geq N\left(S_{r-1,k}^r, \binom{[2,n]}{r}-\binom{[s,n]}{r}\right).
\end{eqnarray*}
Recall that $\nu(\cH-\{1\})<s-1$. By the induction hypothesis, we have $\cH-\{1\}=\binom{[2,n]}{r}-\binom{[s,n]}{r}$. Since $\cL_1(\cH)=\binom{[2,n]}{r-1}$, we have $\cH=\binom{[n]}{r}-\binom{[s,n]}{r}$. This completes the proof.
\end{proof}

\subsection{\normalsize Proof of Theorem \ref{thm1}}
For $k=1$, Theorem \ref{thm1} follows from Theorem \ref{Frankl}. For the case where $k\geq2$, Theorem \ref{thm1.1} implies that the unique extremal hypergraph among all shifted hypergraphs with $\nu(\cH)<s$ is $\binom{[n]}{r}-\binom{[s,n]}{r}$. Next, we prove that up to isomorphism, this is the unique extremal hypergraph with $\nu(\mathcal{H})<s$. In \cite{WaP}, Wang and Peng obtained the following result.

\begin{lem}[Wang and Peng \cite{WaP}]\label{thm1.2}
Let $n\geq 2r+s-2$ and $\cH\subset \binom{[n]}{r}$ is not isomorphic to $\binom{[n]}{r}-\binom{[s,n]}{r}$ but $S_{ij}(\mathcal{H}) = \binom{[n]}{r}-\binom{[s,n]}{r}$, then $\nu(\mathcal{H}) \geq s$.
\end{lem}

The above lemma implies that any $r$-graph $\cH$ satisfying $S_{ij}(\mathcal{H}) = \binom{[n]}{r}-\binom{[s,n]}{r}$ and $\nu(\mathcal{H})<s$ is isomorphic to $\binom{[n]}{r}-\binom{[s,n]}{r}$. Let $n'_0(k,s,r)$ denote the constant given in Theorem~\ref{thm1.1}. Therefore, for all $k\geq 2$ and $n>n_0(k,s,r):=\max\{2r+s-3,n'_0(k,s,r)\}$, Theorem \ref{thm1} follows from Theorem \ref{thm1.1} and Lemma \ref{thm1.2}.

\section{\normalsize Concluding remarks}\label{4}
\smallskip
\noindent $\bullet$ We have determined the maximum number of copies of $S_{r-1,k}^r$ in an $r$-graph with bounded matching number and characterized all extremal hypergraphs (see Theorem \ref{thm1}). Recall that our main tool is the shifting method. In general, for generalized Tur\'an problems on hypergraphs, two conditions are required to apply the shifting method. First, the shifting operation must preserve $\cF$-freeness. Second, it must not decrease the number of copies of the given subhypergraph. Regarding the former condition, we note that the shifting operation preserves $M_s^r$-freeness (see Lemma \ref{lem2.1}), but does not preserve $P^r_\ell$-freeness, $C^r_\ell$-freeness, $S^r_\ell$-freeness, or $K^r_3$-freeness (see Figure \ref{fig1}). Regarding the latter condition, we note that the shifting operation does not decrease the number of copies of $K_s^{(r)}$ (see Lemma \ref{lem2.2}) or of $S_{r-1,k}^r$ for all $k\geq1$ (see Lemma \ref{newlem}), but may strictly decrease the number of copies of $S_{a,k}^r$ for all $1\leq a\leq r-2$ and $k\geq2$ (see Figure \ref{fig2}). This naturally leads to the following two questions.

\begin{que}
For which hypergraphs $\cF$ does the shifting operation preserve $\cF$-freeness in any $r$-graph?
\end{que}

\begin{que}
For which subhypergraphs $\cG$ does the shifting operation not decrease the number of copies of $\cG$ in any $r$-graph?
\end{que}

\bigskip
\noindent $\bullet$ In Theorem \ref{thm1}, we only establish the existence of $n_0(k,s,r)$. In fact, through more delicate calculations by comparing inequalities (3.2) and (3.3) in the proof of Theorem \ref{thm1.1}, one can derive an explicit expression for $n_1(k,s,r)$. As a consequence, an explicit expression for $n_0(k,s,r)$ in Theorem \ref{thm1} can be obtained.

\bigskip
\noindent $\bullet$ When $r=3$, the lower bound on $n$ given by Theorem \ref{Frankl} is $n\ge 5s-2$. Nevertheless, Frankl derived a stronger lower bound $n\ge 3s$ in \cite{Fran2017}. By substituting this stronger bound in place of the one given by Theorem~\ref{Frankl}, one can improve the bound on $n'_0(k,s,3)$ given in Theorem~\ref{thm1.1}. Further, the corresponding bound on $n_0(k,s,3)$ in Theorem \ref{thm1} can also be improved.

\bigskip
\noindent $\bullet$ Very recently, Wu and Zhang \cite{WuZ} established the $t$-intersecting Erd\H{o}s--Ko--Rado theorem in the $(r-1,2)$-norm, which confirms a conjecture of Brooks and Linz \cite{BrL}. Recall that the $t$-intersecting Erd\H{o}s--Ko--Rado theorem states that if $t\leq r\leq n$ and $\cF\subseteq \binom{[n]}{r}$ is $t$-intersecting, then for $n\geq n_0(r,t)$, $|\mathcal{F}|\leq \binom{n-t}{r-t}$.

\begin{thm}[Wu and Zhang \cite{WuZ}]
Let $t,r,n$ be positive integers such that $t\leq r\leq n$. If $\mathcal{F} \subseteq\binom{[n]}{r}$ is $t$-intersecting, then for $n \geq (t + 1)(r - t + 1)$, we have
\begin{eqnarray*}
\|\mathcal{F}\|_{r-1,2}\leq \binom{n-t}{r-t}\big(t + (n - r + 1)(r - t)\big),
\end{eqnarray*}
equality holds if and only if $\mathcal{F}=\big\{ F \in \binom{[n]}{r} : T \subseteq F \big\}$ for some $t$-subset $T$ of $[n]$.
\end{thm}

It is natural to consider the following problem, namely a version of the $t$-intersecting Erd\H{o}s--Ko--Rado theorem in the $(r-1,k)$-norm. In particular, Theorem \ref{cor1} provides a solution to this problem for the special case $t=1$.

\begin{prob}
Let $t,r,n$ be positive integers with $t\leq r\leq n$, and let $\mathcal{F} \subseteq\binom{[n]}{r}$ be a $t$-intersecting family. What is the maximum of $\|\mathcal{F}\|_{r-1,k}$ for $k\geq2$?
\end{prob}

\bigskip
\smallskip
\noindent{\bf{Funding}}
\smallskip


The research of Zhou and Yuan was supported by the National Natural Science Foundation of China (Nos.~12271337 and 12371347).

\bigskip
\noindent{\bf{Declaration of interest}}
\smallskip

The authors declare no known conflicts of interest.



\end{document}